\documentclass[11pt]{amsart}
\usepackage{amssymb,amsmath,amsthm,amscd}
\usepackage{graphicx}
\usepackage{color}
\usepackage[usenames,dvipsnames,svgnames,table]{xcolor}
\usepackage{url}
\usepackage{tikz}
\usepackage{comment}
\usepackage{overpic}
\usepackage{epstopdf}
\usepackage{float}
\usepackage{tikz-cd}
\newtheorem{lemma}{Lemma}[section]
\newtheorem{theorem}[lemma]{Theorem}
\newtheorem{corollary}[lemma]{Corollary}

\theoremstyle{definition}
\newtheorem*{definition}{Definition}
\theoremstyle{remark}
\newtheorem*{remark}{Remark}

\renewcommand{\epsilon}{\varepsilon}
\renewcommand{\phi}{\varphi}

\title[Antiholomorphic perturbations of Weierstrass Zeta functions]{Antiholomorphic perturbations of Weierstrass Zeta functions and Green's function on tori}

\author[K.~Bogdanov]{Konstantin Bogdanov}
\address{Jacobs University Bremen, Campus Ring 1, Bremen, 28759, Germany}
\email{k.bogdanov@jacobs-university.de} 

\author[K. Mamayusupov]{Khudoyor Mamayusupov}
\address{Jacobs University Bremen, Campus Ring 1, Bremen, 28759, Germany}
\email{k.mamayusupov@jacobs-university.de} 

\author[S.~Mukherjee]{Sabyasachi Mukherjee}
\address{Institute for Mathematical Sciences, Stony Brook University, NY, 11794, USA}
\email{sabya@math.stonybrook.edu} 

\author[D.~Schleicher]{Dierk Schleicher}
\address{Jacobs University Bremen, Campus Ring 1, Bremen, 28759, Germany}
\email{d.schleicher@jacobs-university.de} 


\date{\today}

\begin{document}
\begin{abstract}
In \cite{BeEr}, Bergweiler and Eremenko computed the number of critical points of the Green's function on a torus  by investigating the dynamics of a certain family of antiholomorphic meromorphic functions on tori. They also observed that hyperbolic maps are dense in this family of meromorphic functions in a rather trivial way. In this paper, we study the parameter space of this family of meromorphic functions, which can be written as antiholomorphic perturbations of Weierstrass Zeta functions. On the one hand, we give a complete topological description of the hyperbolic components and their boundaries, and on the other hand, we show that these sets admit natural parametrizations by associated dynamical invariants. This settles a conjecture, made in \cite{LW}, on the topology of the regions in the upper half plane $\mathbb{H}$ where the number of critical points of the Green's function remains constant.
\end{abstract} 

\maketitle

\tableofcontents

\section{Introduction}

One of the most fundamental objects in the intersection of potential theory, complex analysis, and mathematical physics is the Green's function. For a proper sub-domain of the Riemann sphere, the Green's function is a fundamental solution of the Laplacian (a harmonic function with a logarithmic singularity at a prescribed point and vanishing on the boundary of the domain) \cite[\S 4.4]{Ra}. From a physical point of view, such a function represents the electric potential of a unit negative charge placed at a point in the interior of the domain. On a compact Riemann surface without boundary, in particular on a torus, there is no fundamental solution of the Laplacian. Indeed, if the surface has no boundary then the flux generated by a unit negative charge must be `absorbed' by some positive charge somewhere on the surface. Hence for a compact surface without boundary, the equation
\begin{equation*}
\Delta G(z,z')=-\delta (z -z')
\end{equation*}
(where $\delta$ is the dirac measure with singularity at $z=z'$) admits no solution. This can also be seen by a straightforward computation (involving Green's formula and integration by parts) as the integral of the left side of the equation over the surface is $0$ while the right side integrates to $-1$. This has the corresponding physical interpretation that for any potential function on a compact surface without boundary, the associated charge distribution must have total integral $0$. This observation leads one to the following generalized Laplace equation on a torus
\begin{equation}
\Delta G(z,z')=-\delta (z -z')+\frac{1}{\vert T\vert},
\label{green_defn}
\end{equation}
where $\vert T\vert$ denotes the area of a torus with respect to a flat metric.

Note that any solution of Equation \ref{green_defn} represents the electric potential of a charge distribution that assigns a unit negative charge at a prescribed point $z'$ and a constant positive charge $\frac{1}{\vert T\vert}$ everywhere else on the torus. By definition, a Green's function on a flat torus is a solution of Equation \ref{green_defn} subject to the normalization $\int_T G(z,z')dA=0$. A more detailed discussion on generalized Laplacian equations can be found in \cite[p. 354]{CH}.

By translation invariance of the Laplacian, we have that $G(z,z')=G(z-z',0)$. Hence it suffices to consider the Green's function $G(z):=G(z,0)$. On a flat torus $\mathbb{C}/\Lambda_\tau$ (where $\Lambda_\tau$ is the lattice generated by $1$ and $\tau$), $G(z)$ can be written explicitly in terms of certain elliptic functions:
\begin{equation}
G(z)=-\frac{1}{2\pi}\log\vert\theta_1(z)\vert+ \frac{(\mathrm{Im}\ z)^2}{2\ \mathrm{Im}\ \tau} +C(\tau),
\end{equation}
where $\theta_1$ is the first theta-function, and $C(\tau)$ is a constant. 

An interesting connection between the Green's function on a flat torus and the mean field equation on a flat torus was discovered a few years ago by Lin and Wang. Instead of formulating the mean field equation precisely, let us just mention that the mean field equation originates from prescribed curvature problems in differential geometry, and has intimate connection with statistical physics. We refer the readers to \cite{LW} (and the references therein) for a rigorous description of the mean field equation and recent progress on its connection with other areas of physics. In \cite[Theorem 1.1]{LW}, Lin and Wang showed that the question of existence of solutions of the mean field equation (on a flat torus) can be completely answered in terms of the number of critical points of the Green's function on the torus. They employed sophisticated PDE techniques to further show that the Green's function on a torus has either three or five critical points, and investigated moduli dependence of the structure of the critical set. 

A fresh and simpler approach was taken by Bergweiler and Eremenko in \cite{BeEr}, where they used techniques of complex dynamics to attack the problem of counting the number of critical points of the Green's function on a flat torus. They investigated a family of antiholomorphic meromorphic functions $\{\widetilde{g}_\tau\}_{\tau\in\mathbb{H}}$ on tori such that the critical points of the Green's function on the torus $\mathbb{C}/\Lambda_\tau$ (where $\Lambda_\tau$ is the lattice generated by $1$ and $\tau$) are precisely the fixed points of $\widetilde{g}_\tau$. The study of these critical points is then turned into a problem about the dynamics of $\widetilde{g}_\tau$ near these fixed points (where `dynamics" refers to iteration). Fundamental results from complex dynamics along with a careful analysis of the dynamical properties of $\widetilde{g}_\tau$ allow them to count the number of fixed points of $\widetilde{g}_\tau$. In this setting, studying moduli dependence of the critical set of the Green's function is equivalent to studying the topology of the parameter space of the family of antiholomorphic meromorphic functions $\{\widetilde{g}_\tau\}_{\tau\in\mathbb{H}}$. 

\begin{figure}[ht!]
\begin{center}
\includegraphics[scale=0.48]{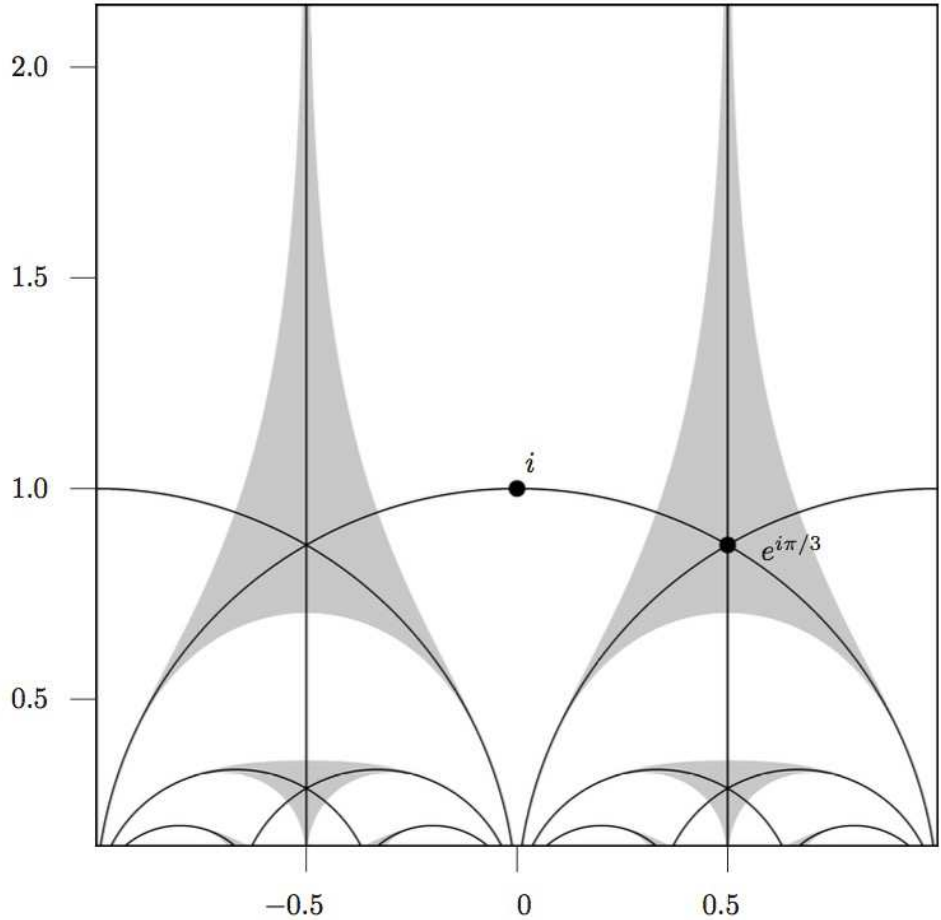}
\caption{The grey and white regions are the hyperbolic components in the parameter space of the family $\{\widetilde{g}_\tau\}_{\tau\in\mathbb{H}}$. Each hyperbolic component is unbounded and simply connected. The boundary of every hyperbolic component consists of two or three simple real-analytic arcs of double parabolic parameters, each of which stretches out to the boundary of $\mathbb{H}$ in both directions. (Figure courtesy Walter Bergweiler and Alexandre Eremenko.)}
\end{center}
\label{parameter_pic}
\end{figure}

A meromorphic map is called \emph{hyperbolic} if the orbit of each of its singular values (i.e. the set of all critical and asymptotic values) converges to an attracting periodic cycle. The map $\widetilde{g}_\tau$ has two singular values (both of which are critical values) counting multiplicities. The set of all parameters $\tau$ in the upper half-plane for which each of the two critical orbits of $\widetilde{g}_\tau$ converges to an attracting cycle is called the \emph{hyperbolic locus} in $\mathbb{H}$. A connected component of the hyperbolic locus is called a hyperbolic component. The importance of hyperbolic components stems from the fact that the qualitative behavior of the dynamics remains stable; i.e. depends continuously on the parameter throughout every hyperbolic component. It is remarkable that the hyperbolic locus of the family $\{\widetilde{g}_\tau\}_{\tau\in\mathbb{H}}$ is dense in the parameter plane $\mathbb{H}$ (in a rather trivial way), a property that is of intrinsic interest in the study of dynamical systems. 

The hyperbolic components of $\{\widetilde{g}_\tau\}_{\tau\in\mathbb{H}}$ are intimately related to the distribution of the number of critical points of Green's functions over the upper half plane $\mathbb{H}$. Following \cite{BeEr}, let $X$ be the set of parameters $\tau$ in $\mathbb{H}$ such that the Green's function on the torus $\mathbb{C}/\Lambda_\tau$ has exactly three critical points, and $Y$ be the set of parameters $\tau$ in $\mathbb{H}$ such that the Green's function on the torus $\mathbb{C}/\Lambda_\tau$ has exactly five critical points. This yields a partition of our parameter space $\mathbb{H}=X\sqcup Y$. It was proved in \cite{BeEr} that $X$ is a closed set in $\mathbb{H}$ (so $Y$ is open), and these two sets share a common boundary in the parameter plane $\mathbb{H}$: i.e. $\partial X=\partial Y$\footnote{Here, and in the rest of the article, for any subset $A$ of the parameter plane $\mathbb{H}$, we will denote the boundary of $A$ in $\mathbb{H}$ by $\partial A$.}. This partition has an important dynamical interpretation. For each $\tau\in Y$, the map $\widetilde{g}_\tau$ has exactly two attracting fixed points (the other three fixed points of $\widetilde{g}_\tau$ are repelling), and each attracting fixed point attracts a singular orbit. On the other hand, for each $\tau$ in $\textrm{int}(X)$ (respectively, on $\partial X$), $\widetilde{g}_\tau$ has a unique attracting fixed point (respectively, a unique parabolic fixed point), whereas the other two fixed points of $\widetilde{g}_\tau$ are repelling (see \cite[\S 8-11]{M1new} for a classification of fixed points for holomorphic maps). Moreover, if $\tau$ is in $\textrm{int}(X)$, then both the singular orbits of $\widetilde{g}_\tau$ converge to the unique attracting fixed point. Therefore, $\textrm{int}(X)\cup Y$ is the hyperbolic locus of our parameter space, and each hyperbolic component is a connected component of $\textrm{int}(X)$ or $Y$.

The principal goal of this paper is to study the topological properties of the hyperbolic components (and their boundaries) in the parameter space of $\{\widetilde{g}_\tau\}_{\tau\in\mathbb{H}}$. The main result is the following (compare Figure \ref{parameter_pic}).

\begin{theorem}[Topology of hyperbolic components and their boundaries]\label{hyperbolic}
Each hyperbolic component $H$ of $\{\widetilde{g}_\tau\}_{\tau\in\mathbb{H}}$ is unbounded and simply connected. The boundary of every hyperbolic component is a union of two or three (according as $H$ is a connected component of $\textrm{int}(X)$ or of $Y$) simple real-analytic arcs of parabolic parameters, each of which stretches out to the boundary of $\mathbb{H}$ in both directions. Furthermore, the hyperbolic components as well as their boundaries admit natural dynamical parametrizations.
\end{theorem}

This proves a conjecture made by Lin and Wang \cite[\S1 p.915]{LW} regarding the topology of the regions in the upper half plane $\mathbb{H}$ where the number of critical points of the Green's function (on a flat torus) remains constant.\footnote{This has been independently proved in \cite{CLW} using much heavier machinery of non-linear PDE and modular forms.} As an application of Theorem \ref{hyperbolic}, we obtain the following result on the parameter dependence of the critical points of Green's function.

\begin{theorem}[Real-analyticity of critical points of Green's function]\label{critical_analytic}
The critical points of the Green's function are real-analytic functions of the parameter $\tau$ on the hyperbolic locus $\mathrm{int}(X)\cup Y$ and on its complement $\partial X$ (the parabolic locus) separately.
\end{theorem}

The paper is organized as follows. In Section \ref{attracting}, we recall how the problem of finding critical points of the Green's function can be turned into a fixed point problem for the antiholomorphic meromorphic maps $\widetilde{g}_\tau$. For an overwhelming majority of parameters in $\mathbb{H}$, the map $\widetilde{g}_\tau$ is hyperbolic, and we associate a conformal conjugacy invariant called \emph{Koenigs ratio} with each hyperbolic map. In Section \ref{interior}, we prove that every hyperbolic component is simply connected, and we use the Koenigs ratio map to give a dynamically natural parametrization of the hyperbolic components. We begin Section \ref{boundary_hyp} with a brief survey of some known facts about parabolic maps in antiholomorphic dynamics. This is one aspect where antiholomorphic dynamics differs from holomorphic dynamics in a rather subtle way. Finally, we use these `parabolic tools' to describe the topology of the boundaries of hyperbolic components (in our parameter space), and show that they admit dynamical parametrizations in terms of suitable conformal invariants. 

It is worth mentioning that antiholomorphic dynamics and associated parameter spaces have been extensively studied in \cite{Na1,NS,HS,MNS,IM1,IM2}.

We would like to thank Alexandre Eremenko and Walter Bergweiler for introducing us to the problem, and for various helpful discussions. Thanks are also due to them for allowing us to reproduce figures from their paper \cite{BeEr}. The first two authors gratefully acknowledge the support of Deutsche Forschungsgemeinschaft DFG during this work.

\section{The family $\widetilde{g}_\tau$, and Koenigs Ratio}\label{attracting}
Let us spend a few words on the connection between the Green's function on a torus and the antiholomorphic meromorphic function $\widetilde{g}_\tau$. By \cite{LW}, the critical points of the Green's function of the torus $\mathbb{C}/\Lambda_\tau$ (where $\Lambda_\tau$ is the lattice generated by $1$ and $\tau$) are solutions of:
\begin{eqnarray}\label{green}
\zeta(z)+az+b\overline{z}=0\ \textrm{(mod}\ \Lambda_\tau),
\end{eqnarray} 
where $\zeta$ is the Weierstrass Zeta function (a meromorphic function on $\mathbb{C}$) corresponding to the lattice $\Lambda_\tau$, and $a$ and $b$ are constants that are uniquely determined by the condition that the left side of Equation (\ref{green}) is $\Lambda_\tau$-periodic. The question of finding solutions of this equation can be turned into a fixed point problem by writing it as:
\begin{eqnarray}\label{fixed}
-\frac{1}{\overline{b}}\left(\overline{\zeta(z)}+\overline{az}\right)=z\ \textrm{(mod}\ \Lambda_\tau).
\end{eqnarray}
This leads us to the antiholomorphic meromorphic map $g_\tau(z):= -\frac{1}{\overline{b}}\left(\overline{\zeta(z)}+\overline{az}\right)$. Therefore, the critical points of the Green's function on $\mathbb{C}/\Lambda_\tau$ are precisely the fixed points of $g_\tau$, modulo $\Lambda_\tau$. The choice of $a$ and $b$ guarantees that 
\begin{eqnarray}\label{commute}
g_\tau(z+\omega)=g_\tau(z)+\omega,\ \forall\ \omega\in\Lambda_\tau.
\end{eqnarray}
However, $g_\tau$ has poles. Let $P_0$ be the set of poles of $g_\tau(z)$. To obtain a dynamical system; i.e. so that we can iterate the map $g_\tau$, we consider the open set $\displaystyle F:=\mathbb{C}\setminus\overline{\bigcup_{n=0}^\infty g_\tau^{-n}(P_0)}$. Thus $g_\tau:F\rightarrow F$ is a dynamical system. Since $g_\tau$ commutes with translations by elements of the lattice $\Lambda_{\tau}$, it descends to a map on the subset $\widetilde{F}$ (projection of $F$) of the torus $\mathbb{C}/\Lambda_\tau$, we call this torus map $\widetilde{g}_\tau$. Since $g_\tau$ is an odd function, $\widetilde{g}_\tau$ commutes with the holomorphic involution $z\mapsto -z$ of the torus. Note that the second iterate $\widetilde{g}_\tau^2$ of the antiholomorphic map $\widetilde{g}_\tau$ is holomorphic. According to \cite[Lemma 1, Lemma 2]{BeEr}, the holomorphic map $\widetilde{g}_\tau^2$ has no asymptotic values, and has two critical points $c_\tau$ and $-c_\tau$. Therefore, the orbits of the singular values of $\widetilde{g}_\tau$ are simply the post-critical orbits $\lbrace\widetilde{g}_\tau^{\circ n}(\pm c_\tau)\rbrace_{n\geq 1}$.

Let us now define a conformal conjugacy invariant for the hyperbolic parameters in our parameter plane. In Section \ref{interior}, we will make use of this conformal invariant to parametrize the hyperbolic components.

Let $H\subset\textrm{int}(X)$ be a hyperbolic component, i.e. for each $\tau \in H$, the map $\widetilde{g}_{\tau}$ has a unique attracting fixed point $z_{\tau}$ (which is a half-period of the lattice). Let us first assume that this attracting fixed point is not super-attracting. Then its immediate basin of attraction (the connected component of the basin of attraction containing the attracting fixed point) must contain a critical point, which is not the attracting point. By symmetry, the other critical point also belongs to the immediate basin. Thus both critical points $c_\tau$ and $-c_\tau$ belong to the immediate basin, and their forward orbits converge to the attracting fixed point $z_\tau$. For $\widetilde{h}_\tau := \widetilde{g}_{\tau}^{\circ 2}$, put $\widetilde{h}_{\tau}'(z_\tau) = \vert\lambda_\tau\vert^2>0$. Let $\kappa_\tau: U_\tau\to \mathbb{C}$ be a holomorphic linearizing (Koenigs) coordinate for $\widetilde{h}_\tau$ defined in $U_\tau$ so that $\kappa_\tau (\widetilde{h}_\tau(z))=\vert\lambda_\tau\vert^2 \kappa_\tau(z)$ (compare \cite[Theorem 8.2]{M1new}). In what follows, we will work with the critical point $c_\tau$. Note that the single critical orbit $\displaystyle\lbrace \widetilde{g}_{\tau}^{\circ n}(c_\tau)\rbrace_{n=1}^\infty$ of $\widetilde{g}_{\tau}$ splits into two critical orbits $\displaystyle\lbrace \widetilde{g}_{\tau}^{\circ 2n}(c_\tau)\rbrace_{n=1}^\infty$ and $\displaystyle\lbrace \widetilde{g}_{\tau}^{\circ 2n-1}(c_\tau)\rbrace_{n=1}^\infty$ of $\widetilde{h}_\tau$. We can look at two representatives of these two critical orbits of $\widetilde{h}_\tau$ in any fundamental domain (of the attracting fixed point), and take their ratio in the Koenigs coordinate. This defines for us a conformal conjugacy invariant:
\begin{align*}
\rho_H(\tau)
&:=\frac{\kappa_\tau(\widetilde{g}_{\tau}(c_\tau))}{\kappa_\tau(c_\tau)}\;.
\end{align*} 
If $\lambda_\tau=0$; i.e. if the unique attracting fixed point of $\widetilde{g}_{\tau}$ is super-attracting, we define $\rho_H(\tau)=0$. For every $\tau\in H$, this invariant $\rho_H(\tau)$ is called the Koenigs ratio of $\tau$.

For hyperbolic components in $Y$, a completely analogous construction is possible. For linearly attracting parameters in $Y$, the two distinct critical orbits (which are related by $-z$) of the antiholomorphic map $\widetilde{g}_{\tau}$ converge to two distinct attracting fixed points (which are related by $-z$ as well) of $\widetilde{g}_{\tau}$. Hence we can work in the immediate basin containing $c_\tau$, and define Koenigs ratio analogously in terms of a linearizing coordinate there. Finally, we define the Koenigs ratio of super-attracting parameters to be $0$.

This ratio is well-defined as the choice of Koenigs coordinate does not affect it. It is easy to verify that $\vert \rho_H(\tau)\vert=\vert\lambda_\tau\vert<1$, and $\vert \rho_H(\tau)\vert\to +1$ as $\tau\to \partial H$. Moreover, since the family of maps $\{\widetilde{g}_\tau\}$ depend real-analytically on the parameter $\tau$, the corresponding Koenigs coordinates $\kappa_\tau$ also depend real-analytically on $\tau$ throughout every hyperbolic component (compare \cite[Theorem 8.2, remark 8.3]{M1new}). Therefore, $\rho_H:H\to\mathbb{D}$ is a real-analytic map, which we will refer to as the \emph{Koenigs ratio} map (compare \cite[\S 6]{IM2}).

\begin{remark}
\begin{enumerate}
\item Our definition of Koenigs ratio agrees with the definition of the `critical value map' introduced in \cite{NS}.

\item Since the involution $z\mapsto -z$ respects the dynamics of  $\widetilde{g}_{\tau}$ , it is easy to see that working with the symmetric critical orbit $\lbrace \widetilde{g}_{\tau}^{\circ n}(-c_\tau)\rbrace$ would define the same map $\rho_H$.
\end{enumerate}
\end{remark}

\section{Simple-connectedness of Hyperbolic Components}\label{interior}

For $\tau\in \textrm{int}(X)\cup Y$, the immediate basins of attraction of every attracting fixed point of $g_{\tau}$ is simply connected. This is easy to see for the centers. In our context, centers are distinguished parameters for which both critical points $c_\tau$ and $-c_\tau$ are fixed by the map $\widetilde{g}_\tau$ (the existence of such parameters was remarked in \cite[p. 2919]{BeEr}). Indeed, in the super-attracting case, the fixed point is the only critical point of $g_{\tau}$ in the immediate basin, and hence the corresponding B{\"o}ttcher coordinate yields a biholomorphism between the immediate basin and the open unit disc (compare \cite[Theorem 9.3]{M1new}). A usual quasiconformal surgery trick (replacing the super-attracting dynamics by a linearly attracting dynamics) now shows that the topology of the Julia set remains stable throughout hyperbolic components (see \cite[\S 4.2]{BF} for the details of the method) . Therefore, for $\tau\in \textrm{int}(X)\cup Y$, the immediate basin of attraction of any fixed point of $g_{\tau}$ is simply connected.

For $\tau\in \textrm{int}(X)$, the restriction of $g_{\tau}$ to the immediate basin of attraction of the unique attracting fixed point is a degree $3$ proper antiholomorphic map (since the immediate basin either contains two simple critical points or a unique double critical point of $g_\tau$). Hence it is conjugate to (via the Riemann map of the immediate basin that sends the attracting fixed point to $0$) a Blaschke product of degree $3$. Since $g_{\tau}(-z)=-g_{\tau}(z)$, the Blaschke products obtained are odd functions; i.e. they are the form $B_{3,a}(z)=\lambda \overline{z}\frac{(\overline{z}-a)(\overline{z}+a)}{(1-\overline{az})(1+\overline{az})}$, with $a\in\mathbb{D}$ and $\vert\lambda\vert=1$, such that $z=1$ is fixed by $B_{3,a}$. The unique such Blaschke product with a super-attracting fixed point is $B_{3,0}$.

For $\tau\in Y$, we will work with the immediate basin of attraction containing the critical point $c_\tau$. The restriction of $g_{\tau}$ to this immediate basin of attraction is a degree $2$ proper antiholomorphic map (since the immediate basin contains a unique simple critical point of $g_\tau$). Hence it is conjugate to (via the Riemann map of the immediate basin that sends the attracting fixed point to $0$) a Blaschke product of degree $2$. Evidently, such a Blaschke product is of the form $B_{2,a}(z)=\lambda\overline{z}\frac{(\overline{z}-a)}{(1-\overline{az})}$, with $a\in\mathbb{D}$ and $\vert\lambda\vert=1$, such that $z=1$ is fixed by $B_{2,a}$. The unique such Blaschke product with a super-attracting fixed point is $B_{2,0}$. 

For $i=2, 3$, let $\mathcal{B}_i$ denote the space of Blaschke products of the form $B_{i,a}$ with $a\in\mathbb{D}$. A direct calculation (or Schwarz lemma) shows that $z=0$ is necessarily an attracting fixed point for every Blaschke product in $\mathcal{B}_i$ (for $i=2, 3$). Clearly, both Blaschke product spaces $\mathcal{B}_2$ and $\mathcal{B}_3$ are simply connected as their common parameter space is the open unit disc $\mathbb{D}$. The above discussion shows that we can associate a unique element of $\mathcal{B}_3$ (respectively, of $\mathcal{B}_2$) to every $g_\tau$ in a hyperbolic component $H$ contained in $\textrm{int}(X)$ (respectively in $Y$). We thus have a map $\eta$ from $H$ to $\mathcal{B}_2$ or $\mathcal{B}_3$. In the rest of this section, we will analyze the topological properties of the map $\eta$, which will in turn reflect on the topology of the hyperbolic components in our parameter space.

For both families of Blaschke products, we can define the Koenigs ratio of the attracting fixed point. The next lemma elucidates the mapping properties of the Koenigs ratio maps defined on $\mathcal{B}_i$, for $i=2,3$.

\begin{lemma}\label{Blaschke}
The Blaschke product model spaces $\mathcal{B}_2$ and $\mathcal{B}_3$ are simply connected. Moreover, the Koenigs ratio map of the attracting fixed point defines a real-analytic branched covering from $\mathcal{B}_2$ (respectively $\mathcal{B}_3$) onto $\mathbb{D}$ of degree $3$ (respectively, of degree $2$), ramified only over the origin.
\end{lemma}
\begin{proof}
Simple connectedness of $\mathcal{B}_2$, and the required properties of the Koenigs ratio map from $\mathcal{B}_2$ onto $\mathbb{D}$ follow from \cite[Lemma 5.4]{NS}.

Although the maps in $\mathcal{B}_3$ are bi-critical, the critical orbits are symmetric with respect to the holomorphic involution $z\mapsto -z$ (i.e. they commute with $-z$), and hence behave similar to maps with a unique free critical point. In particular, they have a unique attracting fixed point, and the orbits of both critical points converge to that attracting fixed point. Moreover, due to the symmetry, the critical points lie on the same equipotential. As a consequence, maps in $\mathcal{B}_3$ are uniquely determined by their associated Koenigs ratio, up to the choice of a fixed point on the unit circle. Note that every $B\in \mathcal{B}_3$ has four fixed points on the unit circle, but they are pairwise related by the involution $z\mapsto -z$. This leaves us with two degrees of freedom. One can now apply the quasiconformal deformation arguments as in the proof of \cite[Lemma 5.4]{NS} verbatim to demonstrate that for every connected component $V$ of $\mathcal{B}_3$, the Koenigs ratio map from $V\setminus \{B_{3,0}\}$ onto $\mathbb{D}^*$ is a covering of degree $2$. By the standard theory of covering maps, one now deduces that $V\setminus \{B_{3,0}\}$ must itself be homeomorphic to the punctured disc $\mathbb{D}^*$ with $B_{3,0}$ corresponding to the puncture. Since $\mathcal{B}_3$ has a unique super-attracting point, it follows that $\mathcal{B}_3$ has only one connected component. Therefore, $\mathcal{B}_3$ is homeomorphic to $\mathbb{D}$: i.e. it is simply connected, and the Koenigs ratio map is a real-analytic branched covering from $\mathcal{B}_3$ onto $\mathbb{D}$ of degree $2$ (since Koenigs ratio determines members of $\mathcal{B}_3$ up to the choice of one from two boundary fixed points), ramified only over the origin.
\end{proof}

\begin{lemma}\label{hyp_Blaschke}
Let $H$ be a hyperbolic component.
\begin{enumerate}
\item If $H\subset\textrm{int}(X)$, then $H$ is homeomorphic to the Blaschke product model space $\mathcal{B}_3$, and there exists a homeomorphism respecting the Koenigs ratio of the attracting fixed point. In particular, the Koenigs ratio map is a real-analytic two-fold cover of the unit disk, ramified only over the origin.

\item If $H\subset Y$, then $H$ is homeomorphic to the Blaschke product model space $\mathcal{B}_2$;  and there exists a homeomorphism respecting the Koenigs ratio of the attracting fixed point. In particular, the Koenigs ratio map is a real-analytic three-fold cover of the unit disk, ramified only over the origin.
\end{enumerate}
\end{lemma}
\begin{proof}
The proof is similar to that of \cite[Theorem 5.6]{NS}, so we only give a sketch. 

\noindent \textbf{Case 1: \boldmath{$H\subset\textrm{int}(X)$}.}\quad
Let $\rho_3:\mathcal{B}_3\to\mathbb{D}$ be the Koenigs ratio map defined on the Blaschke product space $\mathcal{B}_3$. As mentioned earlier, we can associate a Blaschke product in $\mathcal{B}_3$ to each $\tau\in H$. This defines a map $\eta:H\to\mathcal{B}_3$. By construction, the attracting dynamics of $\widetilde{g}_\tau$ and that of the Blaschke product $\eta(\tau)\in\mathcal{B}_3$ are conformally conjugate. Since the Koenigs ratio is a conformal conjugacy invariant, it follows that $\eta$ respects the Koenigs ratio map; i.e. $\rho_H=\rho_3\circ\eta$. Moreover, the Koenigs ratio of the attracting fixed point changes continuously with the parameter in both families, and completely determines both $\tau$ in $H$ and $\eta(\tau)$ in $\mathbb{D}$. This shows that $\eta$ is continuous. Now the quasiconformal surgery step in the proof of \cite[Theorem 5.6]{NS} provides us with a local inverse of $\eta$, and shows that $\eta:H\to\mathcal{B}_3$ is a covering map. By Lemma \ref{Blaschke}, $\mathcal{B}_3$ is simply connected, so $\eta$ must be a homeomorphism. Therefore, $\eta$ is the required homeomorphism that respects the Koenigs ratio map.\\

\begin{figure*}
\begin{center}
\includegraphics[scale=0.32]{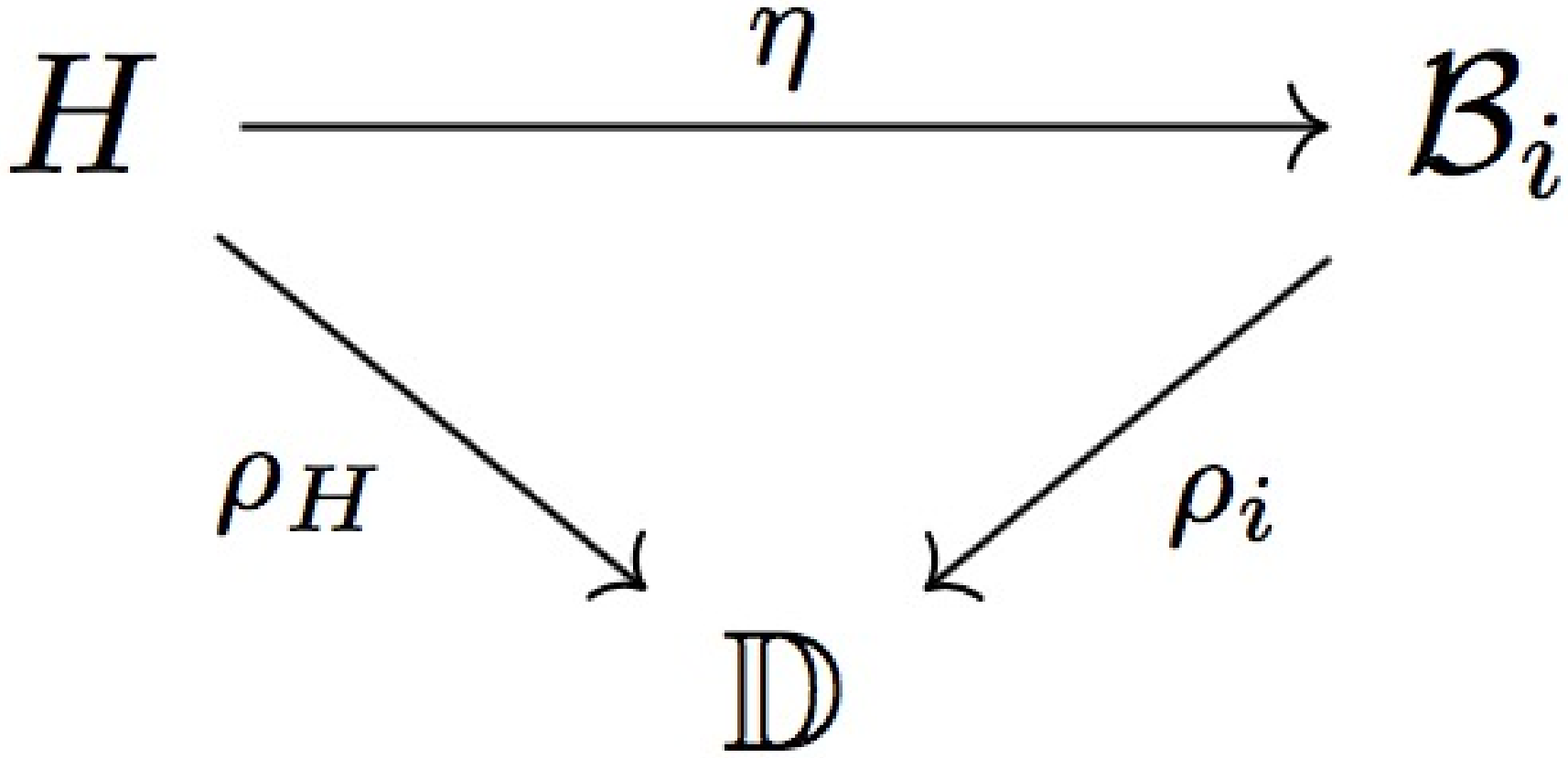}
\end{center}
\end{figure*}

\noindent \textbf{Case 2: \boldmath{$H\subset Y$}.}\quad
Let $\rho_2:\mathcal{B}_2\to\mathbb{D}$ be the Koenigs ratio map defined on the Blaschke product space $\mathcal{B}_2$. The proof goes exactly as in the previous case except that the map $\eta$ sends $H$ homeomorphically (and in a Koenigs ratio preserving manner) onto $\mathcal{B}_2$. 
\end{proof}

\begin{corollary}\label{simply_connected}
Every hyperbolic component is simply connected, and has a unique center. 
\end{corollary}

\section{Boundaries of Hyperbolic Components}\label{boundary_hyp}

We now turn our attention to the boundaries of hyperbolic components in $\mathbb{H}$, namely the set $\partial X=\partial Y$. It was already noted in \cite{BeEr} that the parameter space of the maps $\{\widetilde{g}_\tau\}_{\tau\in\mathbb{H}}$ consists of hyperbolic components, and their common boundaries which are curves of parabolic parameters. Indeed, each parameter on the boundary of a hyperbolic component has an indifferent fixed point. The holomorphic multiplier of a fixed point of an antiholomorphic map is real and positive. Therefore, every indifferent fixed point of an antiholomorphic map must be parabolic with holomorphic multiplier $+1$. Hence, it suffices to discuss the local picture for parabolic fixed points. Note that every non-hyperbolic (necessarily parabolic in our case) parameter in $H$ can be approximated by points in $Y$. Moreover, as we approach such a boundary parabolic point from the interior of a hyperbolic component in $Y$, the two attracting points (which are related by $z\mapsto -z$) come together to merge with a repelling fixed point at a half-period. The resulting parabolic fixed point is thus a double parabolic fixed point of multiplier $+1$; i.e. it has two attracting petals each of which are fixed by $\widetilde{g}_\tau$. We will now rigorously prove this fact.

\begin{lemma}[Indifferent Dynamics of Odd Period]\label{indiff}
Let $\tau_0\in\partial Y$. Then $\widetilde{g}_{\tau_0}$ has a parabolic fixed point. In appropriate local conformal coordinates, the second iterate $\widetilde{g}_{\tau_0}^{\circ 2}$ has the form $z\mapsto z+z^3+\ldots$ near the parabolic fixed point. Moreover, both the attracting petals are fixed by $\widetilde{g}_{\tau_0}$.
\end{lemma}
\begin{proof}
The proof is similar to \cite[Lemma 2.8]{MNS}. Let $H\subset Y$ be a hyperbolic component, and $\tau_0\in\partial H$. Since the parabolic fixed point has multiplier $1$, $\widetilde{g}_{\tau_0}^{\circ 2}$ maps each attracting petal into itself. Hence, each attracting petal contains at least one infinite critical orbit of $\widetilde{g}_{\tau_0}^{\circ 2}$ \cite[Theorem 10.15]{M1new}. However, by \cite[Lemma 2]{BeEr}, $\widetilde{g}_{\tau_0}^{\circ 2}$ has at most four infinite critical orbits. So there can be at most four attracting petals; i.e. $q\leq 5$ (where $q$ is the multiplicity of the parabolic fixed point). But $\widetilde{g}_\tau^{\circ 2}$ commutes with $z\mapsto -z$, and hence we must have an even number of attracting petals. Therefore, the number of attracting petals is either $2$ or $4$; i.e. $q\in\{3,5\}$. 

Now note that for every $\tau\in H$, there are three (repelling) fixed points $\{p_\tau^1, p_\tau^2, p_\tau^3\}$ on the boundary of the immediate basin of attraction $U$ (containing the critical point $c_\tau$) of $\widetilde{g}_{\tau}$. Moreover, these repelling fixed points (on the boundary of the immediate basin of attraction) are half-periods of the corresponding lattice. The parabolic fixed point of $\widetilde{g}_{\tau_0}$ is formed by the merger of the two attracting fixed points and repelling fixed points of $\widetilde{g}_{\tau}$, as $\tau$ tends to $\tau_0$ from the interior of $H$. In order to have a parabolic fixed point with four attracting petals, two distinct repelling fixed points must merge in the dynamical plane of $\widetilde{g}_{\tau_0}$. But this is impossible since two distinct half-periods can never coincide for a lattice. This rules out the case $q=5$. Finally, since the two immediate parabolic basins of $\widetilde{g}_{\tau_0}$ are limits of two fixed basins of attraction of $\widetilde{g}_{\tau}$ (as $\tau$ tends to $\tau_0$ from the interior of $H$), it follows that these petals are fixed by $\widetilde{g}_{\tau_0}$.
\end{proof}

In holomorphic dynamics, the local dynamics in attracting petals of parabolic periodic points is well-understood: there is a local coordinate $\zeta$ which conjugates the first-return dynamics to the form $\zeta\mapsto\zeta+1$ in a right half plane \cite[Section~10]{M1new}. Such a coordinate $\zeta$ is called a \emph{Fatou coordinate}. Thus the quotient of the petal by the dynamics is isomorphic to a bi-infinite cylinder, called an \emph{Ecalle cylinder}. Note that Fatou coordinates are uniquely determined up to addition of a complex constant. 

In antiholomorphic dynamics, an additional structure is given by the antiholomorphic intermediate iterate, and there is a natural choice for a preferred Fatou coordinate. 

\begin{lemma}[Antiholomorphic Fatou Coordinates]\label{PropFatouCoord} 
Let $\tau\in\partial Y$, and $z_0$ be a double parabolic fixed point of $\widetilde{g}_\tau$. Let $U$ be the fixed Fatou component of $\widetilde{g}_\tau$ containing $c_\tau$ such that $z_0 \in \partial U$. Then there is an open subset 
$V \subset U$ with $z_0 \in \partial V$ and $\widetilde{g}_\tau (V) \subset V$ so 
that for every $z \in U$, there is an $n \in \mathbb{N}$ with $\widetilde{g}_\tau^{\circ n}(z)\in 
V$. Moreover, there is a univalent map $\Phi_\tau \colon V \to \mathbb{C}$ with 
$\Phi_\tau(\widetilde{g}_\tau(z)) = \overline{\Phi_\tau(z)}+1/2$, and $\Phi_\tau(V)$ contains a right 
half plane. This map $\Phi_\tau$ is unique up to horizontal translation. 
\end{lemma}

\begin{proof}
The proof is similar to that of \cite[Lemma 2.3]{HS} or \cite[Lemma 3.1]{MNS}. In fact, the arguments apply more generally to antiholomorphic indifferent periodic points such that the attracting petal has odd period.
\end{proof}

The map $\Phi_\tau$ will be called a  \emph{preferred Fatou coordinate} for 
the petal $V$; it satisfies $\Phi_\tau(\widetilde{g}_\tau^{\circ 2}(z))=\Phi_\tau(z)+1$ for $z \in V$ 
in accordance with the standard theory of holomorphic Fatou coordinates. 
 
The antiholomorphic iterate interchanges both ends of the Ecalle cylinder, so it must fix a horizontal line around this cylinder (the \emph{equator}). The change of coordinate has been so chosen that the equator is the projection of the real axis. We will call the vertical Fatou coordinate the \emph{Ecalle height}. Its origin is the equator. The Ecalle height of the critical value $\widetilde{g}_{\tau}(c_{\tau})$ is called the critical Ecalle height of the map $\widetilde{g}_\tau$. This is an important conformal conjugacy invariant of the parabolic maps under consideration.

\begin{lemma}[Parabolic Arcs]\label{ThmParaArcs} 
Let $\tau_0\in\partial Y$. Then $\tau_0$ is on a parabolic arc in the following sense: there  exists a simple real-analytic arc of parabolic
parameters $\tau(t)$ (for $t\in\mathbb{R}$) with quasi-conformally equivalent but conformally 
distinct dynamics of which $\tau_0$ is an interior point.
\end{lemma}

\begin{proof}
The situation is similar to that in \cite[Theorem 3.2]{MNS} (compare \cite[Proposition 2.5, Figure 2.6]{HS}), where a quasi-conformal deformation trick was used to prove the existence of the required arcs. We will follow the same scheme here.

By Lemma \ref{indiff}, $\widetilde{g}_{\tau_0}$ has a double parabolic fixed point such that both of its attracting petals are invariant under $\widetilde{g}_{\tau_0}$. Let $U$ be the immediate basin of attraction (of the parabolic fixed point) containing $c_{\tau_0}$. We choose our preferred Fatou coordinate $\Phi_{\tau_0}$ for $U$ such that it maps the equator (of $U$) to the real line and the image of the critical value $\widetilde{g}_{\tau_0}(c_{\tau_0})$ under $\Phi_{\tau_0}$ has real part $1/4$ within the Ecalle cylinder. Let the critical Ecalle height of $\widetilde{g}_{\tau_0}$ be $h$.

It is easy to change the complex structure within the Ecalle cylinder 
so that the critical Ecalle height becomes any assigned real value, 
for example via the quasi-conformal homeomorphism 
$$
\ell_t : (x,y)\mapsto \left\{\begin{array}{ll}
                    (x,y+2tx) & \mbox{if}\ 0\leq x\leq 1/4, \\
                    (x,y+t(1-2x)) & \mbox{if}\ 1/4\leq x\leq 1/2, \\ 
                    (x,y-t(2x-1)) & \mbox{if}\ 1/2\leq x\leq 3/4, \\ 
                    (x,y-2t(1-x)) & \mbox{if}\ 3/4\leq x\leq 1. 
                      \end{array}\right. 
$$
This homeomorphism $\ell_t$ commutes  with the map $I:z\mapsto\bar{z}+1/2$, 
hence the corresponding Beltrami form is invariant under the  map $I$. Note 
that $\ell_t(1/4,h) = (1/4,h+t/2)$. Translating the map $\ell_t$ by 
positive integers, we obtain a qc-map $\ell_t$ commuting with $I$ in a 
right half plane. 

By the coordinate change $z\mapsto \Phi_{\tau_0}(z)$, we can transport this Beltrami form from the right half plane into the domain of definition of $\Phi_{\tau_0}$ (an attracting petal in $U$), and it is forward invariant under $g_{\tau_0}$. We can use the involution $z\mapsto -z$ to spread this Beltrami form in the symmetric Fatou component $-U$ (this process is compatible with the dynamics since $\widetilde{g}_{\tau_0}$ commutes with the involution). It is easy to make it backward invariant by pulling it back along the dynamics. Extending it by the zero Beltrami form outside of the entire parabolic basin, we obtain a $\widetilde{g}_{\tau_0}$-invariant Beltrami form on the torus. The Measurable Riemann Mapping Theorem now supplies a qc-map $\phi_t$ integrating this Beltrami form. Note that $\phi_t$ is a quasiconformal homeomorphism of the topological torus. It follows that there is some $\tau(t)\in\mathbb{H}$ such that $\phi_t\circ\widetilde{g}_{\tau_0}\circ\phi_t^{-1}:\mathbb{C}/\Lambda_{\tau(t)}\rightarrow\mathbb{C}/\Lambda_{\tau(t)}$ is an antiholomorphic meromorphic function in the same family $\{\widetilde{g}_{\tau}\}_{\tau\in\mathbb{H}}$. Hence, $\phi_t\circ\widetilde{g}_{\tau_0}\circ\phi_t^{-1}=\widetilde{g}_{\tau(t)}$. Its attracting Fatou coordinate at the parabolic point is given by $\Phi_t = \ell_t\circ\Phi_{\tau_0}\circ\phi_t^{-1}$. Thus the critical Ecalle height of $\widetilde{g}_{\tau(t)}$ is $h+t/2$. 

Note that the Beltrami form depends real-analytically on $t$, so the 
parameter $\tau(t)$ depends real-analytically on $t$. We thus obtain a real
analytic map from $\mathbb{R}$ into our parameter space $\mathbb{H}$. Since the critical
values of all $\widetilde{g}_{\tau(t)}$ have  different Ecalle heights, which is a
conformal invariant, this map is injective. This proves the existence of an arc in the parameter plane with the required properties, and injectivity implies that the arc is simple.
\end{proof}

\begin{remark}
Such a real-analytic arc $\mathcal{C}$ consisting of parabolic parameters is called a \emph{parabolic arc}.
\end{remark}

\begin{corollary}
The non-hyperbolic locus $\partial Y$ is a union of parabolic arcs.
\end{corollary}

\begin{lemma}\label{union_arcs}
Let $H$ be a hyperbolic component. Then $\partial H$ is a union of parabolic arcs. 
\end{lemma}
\begin{proof}
 By Lemma \ref{indiff}, $\partial H$ consists of double parabolic parameters. By Theorem \ref{ThmParaArcs}, each parameter on $\partial H$ lies on a parabolic arc, and all the parameters on a given parabolic arc are quasiconformally conjugate. These double parabolic parameters are formed by the merger of two attracting fixed points and a repelling fixed point $p_\tau^i$, where $i\in\{1,2,3\}$ is constant along a given parabolic arc. If two such parabolic arcs meet at a point in $\mathbb{H}$, then the resulting parabolic parameter would have four attracting petals (it would have a parabolic fixed point of multiplicity $5$), which contradicts Lemma \ref{indiff}. Therefore, two parabolic arcs do not intersect. Since $\partial H$ is connected, it follows that any parabolic arc that intersects $\partial H$ must, in fact, be contained in $\partial H$. We can now conclude that $\partial H$ is a union of parabolic arcs.
\end{proof}

\begin{proof}[Proof of Theorem \ref{hyperbolic}]
Let $H$ be a hyperbolic component. We have already seen that $H$ is simply connected (Corollary \ref{simply_connected}), and $\partial H$ is a union of parabolic arcs (Lemma \ref{union_arcs}). We will now proceed to counting the number of parabolic arcs on $\partial H$.

\noindent \textbf{Case 1: \boldmath{$H\subset\textrm{int}(X)$}.}\quad
We need to understand the dynamics of parabolic parameters as limiting dynamics of hyperbolic parameters in $H$. This can be conveniently achieved by looking at the holomorphic second iterate $g_\tau^{\circ 2}$, where $\tau\in H$. Since the degree of $g_\tau^{\circ 2}$ restricted to an immediate basin of attraction is $3^2=9$, it must have eight fixed points on the boundary of the immediate basin. Four of these are fixed points of $g_\tau$ (note that these are all half-periods, and they project to only two distinct fixed points of $\widetilde{g}_\tau$ on the torus). It follows that the rest of the four fixed points of $g_\tau^{\circ 2}$ come from a pair of $2$-cycles of $g_\tau$ (this holds for $\widetilde{g}_\tau$ as well). Let us denote these two $2$-cycles of $\widetilde{g}_\tau$ by $\{q_\tau^1, \widetilde{g}_\tau(q_\tau^1)\}$ and $\{q_\tau^2, \widetilde{g}_\tau(q_\tau^2)\}$. As $\tau$ approaches a parabolic arc on $\partial H$, one of these two $2$-cycles of $\widetilde{g}_\tau$ merge with the unique attracting fixed point giving rise to a double parabolic fixed point. Recall that all maps on a given parabolic arc are quasiconformally conjugate. Hence, given any parabolic arc $\mathcal{C}\subset\partial H$, there is a fixed $i\in\{1,2\}$ such that as $\tau$ approaches $\mathcal{C}$, the $2$-cycle $\{q_\tau^i, \widetilde{g}_\tau(q_\tau^i)\}$ of $\widetilde{g}_\tau$ merges with its attracting fixed point to produce a double parabolic fixed point. This discussion shows that there are two distinct ways in which a double parabolic parameter can be formed on $\partial H$. It follows that there are two parabolic arcs $\mathcal{C}_1$ and $\mathcal{C}_2$ on $\partial H$ satisfying the property that the double parabolic fixed point of every $\widetilde{g}_{\tau}$ on $\mathcal{C}_i$ (where $i\in\{1,2\}$) is formed by the merger of an attracting fixed point and the repelling $2$-cycle $\{q_\tau^i, \widetilde{g}_\tau(q_\tau^i)\}$.

\begin{figure}[ht!]
\includegraphics[scale=0.282]{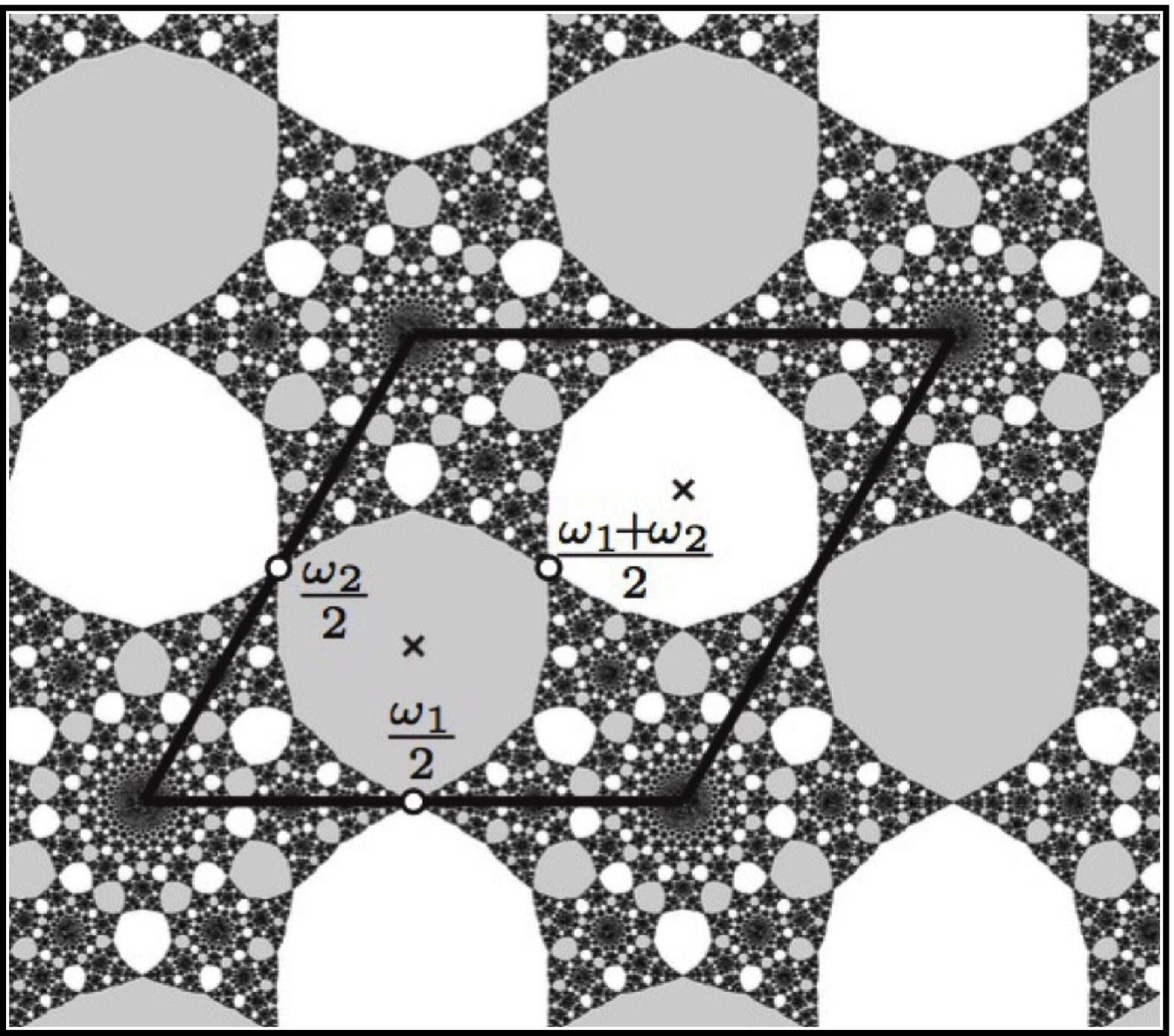}\ \includegraphics[scale=0.342]{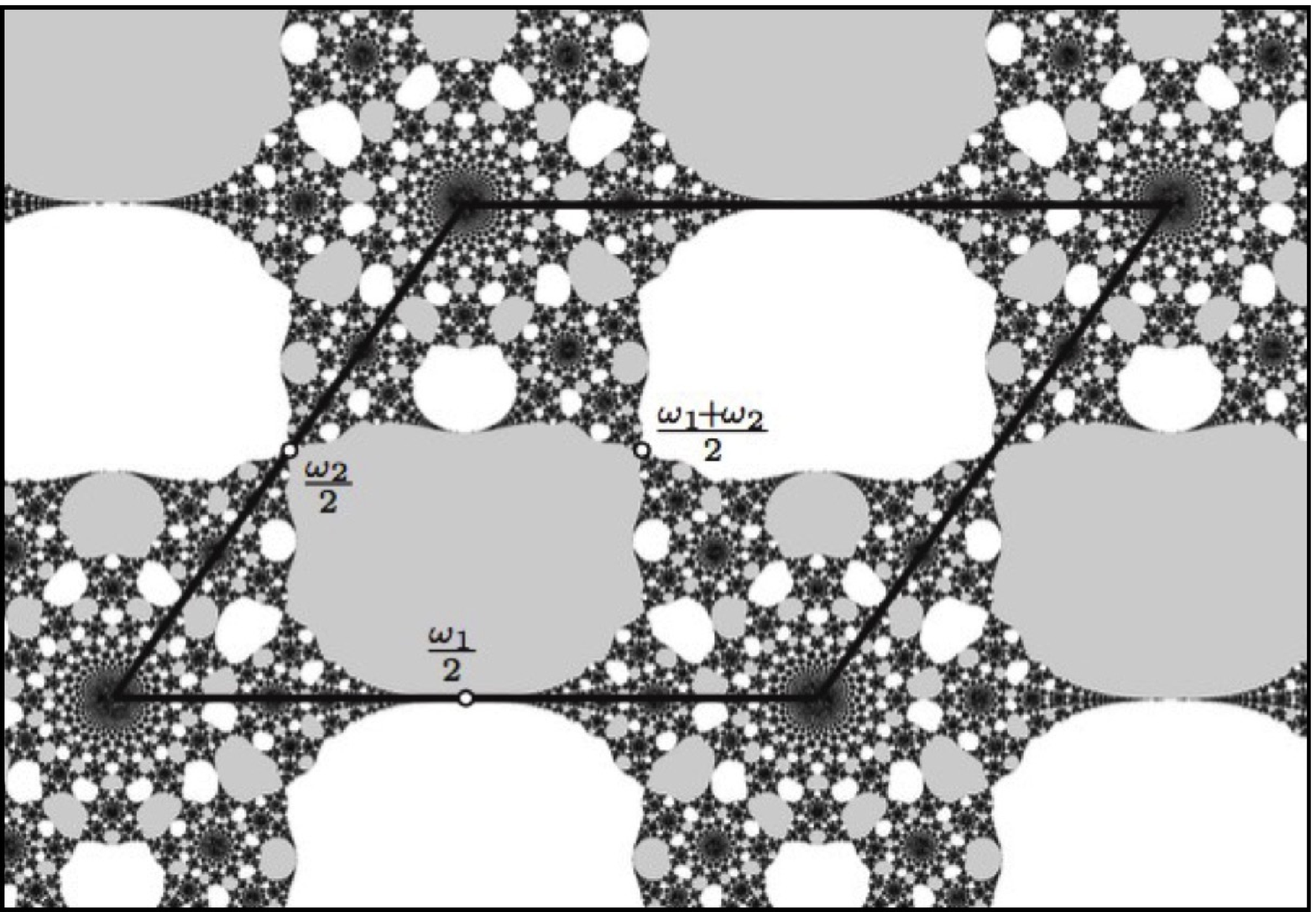}
\caption{Left: A fundamental domain in the dynamical plane of a map $g_\tau$, where $\tau$ belongs to a hyperbolic component contained in $Y$. The two attracting fixed points in the fundamental domain are marked by crosses, and the three repelling fixed points (which are half-periods of the lattice) are circled. As one moves closer to the boundary of the hyperbolic component, the two attracting points tend to coalesce with a repelling fixed point. Right:  A fundamental domain in the dynamical plane of a map $g_\tau$, where $\tau$ belongs to the boundary of a hyperbolic component. The parabolic fixed point, which is at $\frac{\omega_1+\omega_2}{2}$, is formed by the merger of two attracting fixed points and a repelling fixed point. The other two fixed points at $\frac{\omega_1}{2}$ and $\frac{\omega_2}{2}$ are repelling. (Figure courtesy Walter Bergweiler and Alexandre Eremenko.)}
\label{merger_pic}
\end{figure}

\noindent \textbf{Case 2: \boldmath{$H\subset Y$}.}\quad
The idea of the proof is similar to that of Case 1 (also compare \cite[Theorem 1.2]{MNS}). As mentioned earlier, for every $\tau\in H$, there are three (repelling) fixed points $\{p_\tau^1, p_\tau^2, p_\tau^3\}$ on the boundary of the immediate basin of attraction $U$ (containing the critical point $c_\tau$) of $\widetilde{g}_{\tau}$. By Lemma \ref{union_arcs}, $\partial H$ is a union of parabolic arcs. For every fixed parabolic arc $\mathcal{C}$ on $\partial H$, there exists $i\in\{1,2,3\}$ such that as $\tau$ approaches $\mathcal{C}$, the attracting fixed points of $\widetilde{g}_{\tau}$ merge with the repelling fixed point $p_\tau^i$ (compare Figure \ref{merger_pic}). Since there are three choices for the boundary fixed point $p_\tau^i$ with which the attracting points can merge (i.e. there are three distinct ways in which a parabolic point can be born), it follows that there are three parabolic arcs $\mathcal{C}_1$, $\mathcal{C}_2$, and $\mathcal{C}_3$ on $\partial H$ satisfying the property that the parabolic fixed point of any $\widetilde{g}_{\tau}$ on $\mathcal{C}_i$ (where $i\in\{1,2,3\}$) is formed by the merger of two attracting fixed points and the repelling fixed point $p_\tau^i$. 

Finally, $\partial H$ is connected, and each $\mathcal{C}_i$ is an open simple arc. It follows that any accumulation point of $\mathcal{C}_i (\subset\partial H)$ in $\mathbb{H}$ must also be an accumulation point of some $\mathcal{C}_j (\subset\partial H)$ with $j\neq i$. But such an accumulation point has to be a parabolic parameter (on $\partial H$) with multiplicity $4$, which contradicts Lemma \ref{indiff}. Hence no parabolic arc on $\partial H$ has an accumulation point in $\mathbb{H}$; i.e. $\left(\overline{\mathcal{C}_i}\setminus\mathcal{C}_i\right)\cap\mathbb{H}=\emptyset$ (here, $\overline{\mathcal{C}_i}$ denotes the topological closure of $\mathcal{C}_i$ in the complex plane $\mathbb{C}$). In other words, they stretch out to the boundary of $\mathbb{H}$ in both directions. Unboundedness of $H$ readily follows. Since every hyperbolic component can be parametrized by the Koenigs ratio function, and each parabolic arc can be parametrized by Ecalle height, the last statement of the theorem is clear. 
\end{proof}

The proof of Theorem \ref{critical_analytic} is now straightforward.

\begin{proof}[Proof of Theorem \ref{critical_analytic}]
For every parameter $\tau$ in the hyperbolic locus $\mathrm{int}(X)\cup Y$, the fixed points of $\widetilde{g}_\tau$ are either attracting or repelling. An easy application of the implicit function theorem shows that such a fixed point can be locally followed as a real-analytic function of the parameter (compare \cite[Lemma 2.2]{S1a}). On the other hand, by Lemma \ref{ThmParaArcs}, variation of critical Ecalle height yields a quasiconformal conjugacy between any two maps on a given parabolic arc, and these conjugacies depend real-analytically on the parameter $\tau$. Hence, the fixed points of $\widetilde{g}_\tau$ are real-analytic functions of the parameter on $\partial X$.

Since the fixed points of $\widetilde{g}_\tau$ are precisely the critical points of the Green's function on the flat torus $\mathbb{C}/\Lambda_\tau$, the critical points of the Green's function turn out to be real-analytic functions of the parameter $\tau$ on the hyperbolic locus $\mathrm{int}(X)\cup Y$ and on its complement $\partial X$ (the parabolic locus) separately.
\end{proof}

We conclude with the landing behavior of internal rays of hyperbolic components.

\begin{definition}[Internal Rays of Hyperbolic Components]
An internal ray of a hyperbolic component $H$ is an arc $\gamma\subset H$ starting at
the center such that there is an angle $\theta$ with $\rho_H(\gamma)=\{re^{2\pi i\theta}: r\in[0,1)\}$.
\end{definition}

\begin{remark}
Since $\rho_H$ is a many-to-one map, an internal ray of $H$ with a given
angle is not uniquely defined. In fact, if $H\subset\textrm{int}(X)$, there are two internal rays with any given angle $\theta$, and if $H\subset Y$, there are three internal rays with any given angle $\theta$.
\end{remark}

\begin{lemma}
The internal rays at angle $0$ land at the critical Ecalle height 0 parameters on $\partial H$ (one ray on each
parabolic arc). All other internal rays stretch out to the boundary of $\mathbb{H}$.
\end{lemma}
\begin{proof}
The proof is similar to that of \cite[Lemma 6.2]{IM2}. The proof given there can be adapted to show that the internal rays at angle $0$ land at the Ecalle height $0$ parameters, and no internal ray at a non-zero angle can accumulate on a parabolic arc. Any accumulation point (in $\mathbb{C}$) of an internal ray lies either on $\partial H$ (the boundary of $H$ in the upper half plane), or on the boundary of the upper half plane. But $\partial H$ is a finite union of parabolic arcs, so it follows that the internal rays at non-zero angles do not have any accumulation point in $\partial H$. Hence they must accumulate on the boundary of the upper half plane. This completes the proof of the lemma.
\end{proof}

\bibliographystyle{alpha}
\bibliography{torus}
 
\end{document}